\documentclass [a4paper] {article}
\usepackage {a4wide}
\usepackage {amssymb}
\usepackage {amsmath}
\usepackage {multirow, array}
\usepackage [english] {babel}
\usepackage {theorem}
\usepackage {enumerate}
\usepackage {arydshln}
\usepackage {graphicx}
\usepackage[utf8]{inputenc}
\usepackage {hyperref}
\usepackage {tikz}

\newtheorem{dfn}{Definition}[section]
\newtheorem{theo}[dfn]{Theorem}
\newtheorem{prop}[dfn]{Proposition}
\newtheorem{cor}[dfn]{Corollary}

\newtheorem{rem}[dfn]{Remark}
\newtheorem{lem}[dfn]{Lemma}

\newtheorem{question}[dfn]{Question}
\newtheorem{obs}[dfn]{Observation}

\newenvironment{proof}{
\par
\noindent {\bf Proof.}\rm}{\mbox{}\hfill$\square$\par\vskip 3mm}

\newcommand{\R}{\mathcal{R}}
\renewcommand{\L}{\mathcal{L}}
\newcommand{\N}{\mathcal{N}}
\renewcommand{\P}{\mathcal{P}}
\newcommand{\delete}[1]{}
\newcommand{\Rset}[1]{{#1}^{\boldsymbol{R}}}
\newcommand{\Lset}[1]{{#1}^{\boldsymbol{L}}}
\newcommand{\GR}{\Rset{G}}
\newcommand{\GL}{\Lset{G}}
\newcommand{\HR}{\Rset{H}}
\newcommand{\HL}{\Lset{H}}
\newcommand{\UU}{\mathcal{U}}
\newcommand{\BB}{\mathcal{B}}
\newcommand{\DD}{\mathcal{D}}
\newcommand{\DDI}{{\mathcal{D}_I}}

\newcommand{\II}{\mathcal{I}}
\newcommand{\NN}{\mathbb{N}}
\newcommand{\ab}{{\merge{\DD}{\BB}}}

\newcommand{\birth}{{\rm b}}
\newcommand\merge[2]{%
  \ooalign{\hfil$\vcenter{\hbox{$#1$}}$\hfil\cr
    \hfil$\vcenter{\hbox{$\scriptstyle #2$}}$\hfil}}

\title{Binary dicots, a core of dicot games}

\delete{
\author{Gabriel Renault${}^{1,2}${}\footnote{email: gabriel.renault@labri.fr}
\vspace{.3cm} \\
${}^{1}$ Univ. Bordeaux, LaBRI, UMR5800, F-33400 Talence\\
${}^{2}$ CNRS, LaBRI, UMR5800, F-33400 Talence
}
}

\author{Gabriel Renault{}\footnote{email: gabriel.renault@labri.fr}
\vspace{.3cm} \\
Univ. Bordeaux, LaBRI, UMR5800, F-33400 Talence\\
CNRS, LaBRI, UMR5800, F-33400 Talence \\
Department of Mathematics, Beijing Jiaotong University, Beijing 100044 P. R. China
}

\begin{document}

\maketitle

\begin{sloppypar}

\begin{abstract}
We study combinatorial games under mis\`ere convention.
Several sets of games have been considered earlier to better understand the behaviour of mis\`ere games.
We here connect several of these sets.
In particular, we prove that comparison modulo binary dicot games is often the same as comparison modulo dicot games, and that equivalence modulo dicot games and modulo impartial games are the same when they are restricted to impartial games.
\end{abstract}

\section{Introduction}
\label{sec:intro}

In this paper, we study combinatorial games under mis\`ere convention, and focus our analyze on some special families of games, namely dicot games, binary games, impartial games and their intersections.
We first recall basic definitions, following~\cite{lip, ww, onag, cgt}.

A combinatorial game is a finite two-player game with no chance and perfect information.
The players, called Left and Right\footnote{By convention, Left is a female player whereas Right is a male player.}, alternate moves until one player is unable to move.
The last player to move wins the game in its normal version, while that player would lose the game in its mis\`ere version.
In this paper, we are mostly considering games in their mis\`ere version.

A game can be defined recursively by its sets of options $G = \{\GL|\GR\}$, where $\GL$ is the set of games Left can reach in one move (called Left options), and $\GR$ the set of games Right can reach in one move (called Right options).
The typical Left option of $G$ is denoted $G^L$, and the typical Right option of $G$ is denoted $G^R$.
A follower of a game $G$ is a game that can be reached from $G$ after a succession of (not necessarily alternating) Left and Right moves.
Note that a game $G$ is considered one of its own followers.
The zero game $0=\{\cdot|\cdot\}$, is the game with no options (a dot indicates an empty set of options).
A Left end (resp. Right end) is a game where Left (resp. Right) cannot move.
The birthday $\birth(G)$ of a game $G$ is defined recursively as one plus the maximum birthday of the options of $G$, with $0$ being the only game with birthday $0$.
For example, the game $*=\{0|0\}$ has birthday $1$.

The (disjunctive) sum $G + H$ of two games $G$ and $H$ is defined recursively as \mbox{$G+H = \{\GL+H,G+\HL|\GR+H,G+\HR\}$}, where $\GL+H$ is understood to range over all sums of $H$ with an element of $\GL$, that is the game where each player can on their turn play a legal move for them in one (but not both) of the components.
The conjugate $\overline{G}$ of a game $G$ is recursively defined as $\overline{G} = \{\overline{\GR}|\overline{\GL}\}$, where again $\overline{\GR}$ is understood to range over all conjugates of elements of $\GR$, that is the game where Left's and Right's roles are reversed.

Under both conventions, we can sort all games into four sets, depending on their outcomes.
When Left has a winning strategy on a game $G$ no matter which player starts, we say $G$ has outcome $\L$, and $G$ is an $\L$-position.
Similarly, $\N$, $\P$ and $\R$ (for Next, Previous and Right) denote respectively the outcomes of games on which the first player, the second player and Right has a winning strategy whoever starts the game.
The mis\`ere outcome of a game $G$ is denoted $o^-(G)$, while its normal outcome is denoted $o^+(G)$.
Outcomes are partially ordered according to Figure~\ref{fig:order}, with Left prefering greater games.

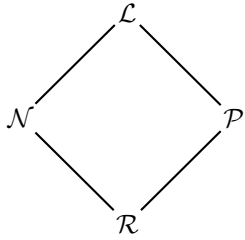
\begin{figure}
\begin{center}

\begin{tikzpicture}
[thick,scale=1,
     vertex/.style={circle,draw=white!100,inner sep=1pt,minimum
size=2mm,fill=white!100},
     blackvertex/.style={circle,draw,inner sep=0pt,minimum
size=1.5mm,fill=black!100},
     clause/.style={circle,draw,inner sep=0pt,minimum
size=3mm,fill=white!100}]


\coordinate (L) at (0,2.8);
\coordinate (N) at (-1.4,1.4);
\coordinate (P) at (1.4,1.4);
\coordinate (R) at (0,0);

\draw (L)--(N)--(R)--(P)--(L);

\draw (L) node[vertex] {$\L$};
\draw (N) node[vertex] {$\N$};
\draw (R) node[vertex] {$\R$};
\draw (P) node[vertex] {$\P$};

\end{tikzpicture}
\vspace{-0.7cm}
\end{center}
\caption{Partial ordering of outcomes}\label{fig:order}
\end{figure}

A game $G$ is said to be greater than or equal to a game $H$ in mis\`ere play whenever Left always prefer having the game $G$ to the game $H$ in a sum, that is $G \geqslant^- H$ if we have \mbox{$o^-(G+X) \geqslant o^-(H+X)$} for every game $X$.
A game $G$ is said to be equivalent to a game $H$ in mis\`ere play, denoted $G \equiv^- H$, if we have both $G \geqslant^- H$ and $H \geqslant^- G$.
Two games $G$ and $H$ are said to be incomparable if we have neither $G \geqslant^- H$ nor $H \geqslant^- G$.
Comparability, equivalence and incomparability are defined similarly in normal play, using superscript $+$ rather than $-$.

Mesdal and Ottaway~\cite{partizan}, and Siegel~\cite{canonical} gave evidence that equivalence and even comparability are very limited in general mis\`ere play.
This is why Plambeck and Siegel defined in~\cite{quotient1,quotient2} equivalence modulo restricted sets of games, leading to a breakthrough in the study of mis\`ere play games.

\begin{dfn}[\cite{quotient1,quotient2}]
{\rm
Let $\UU$ be a set of games, $G$ and $H$ two games (not necessarily in $\UU$). 
We say $G$ is greater than or equal to $H$ modulo $\UU$ in mis\`ere play and write 
$G \geqslant^- H \pmod{\UU}$ 
if $o^-(G+X) \geqslant o^-(H+X)$ for every $X \in \UU$. 
We say $G$ is equivalent to $H$ modulo $\UU$ in mis\`ere play and write $G \equiv^- H \pmod{\mathcal{U}}$ if $G \geqslant^- H \pmod{\mathcal{U}}$ and $H \geqslant^- G \pmod{\mathcal{U}}$.
}
\end{dfn}

For instance, Plambeck and Siegel~\cite{quotient1,quotient2} considered the sets of all positions of given games, octal games in particular.
Other sets have been considered, including the sets of alternating games $\mathcal{A}$~\cite{alternating}, impartial games $\mathcal{I}$~\cite{ww,onag}, dicot games $\mathcal{D}$~\cite{mpg,dicotaxo,sprigs}, dead-ending games $\mathcal{E}$~\cite{pkayles,deadending}, and all games $\mathcal{G}$~\cite{partizan,canonical}.

We believe that having some properties, namely being closed under followers, addition and conjugates, makes a set more relevant to be studied.
We hence define a universe to be a set closed under followers, addition and conjugates.
When a set $\UU$ is not a universe, it is natural to consider the closure $c\ell(\UU)$ of $\UU$, that is the smallest set containing $\UU$ that is closed under addition and followers.
Note that $c\ell(\UU)$ might still not be closed under conjugates.

To simplify notations, we use $\geqslant^-_\UU$ and $\equiv^-_\UU$ to denote superiority and equivalence between games modulo a set $\UU$.
The symbol $=$ between games is here reserved to denote recursive equality on the sets of options.
Observe also that when $\UU$ and $\UU'$ are two sets such that $\UU' \subseteq \UU$, then we have $G\geqslant^-_{\UU'} H$ whenever we have $G \geqslant^-_\UU H$.

In the following, we study several sets of games, namely dicot games, binary games, impartial games and their intersections.
A game is said to be dicot if it is $\{\cdot|\cdot\}$ or it has both Left and Right options and all these options are dicot.
A game is said to be binary if it has at most one Left option and at most one Right option, and all these options, if any, are binary.
A game is said to be impartial if its Left options and its Right options are the same, and all these options, if any, are impartial.
Throughout this paper, the universe of dicot games is denoted $\DD$, the set of binary dicot games is denoted $\ab$, and the universe of impartial games is denoted $\II$.
Note that binary games, and binary dicot games are not closed under addition.

Binary dicot games were introduced and studied by Allen in~\cite{mpg,pmq}.
In particular, she proved the invertibility of an infinite family of binary dicot games modulo dicot games.

In the following, outcomes of games (or sums of games) with small birthday are often given without proof, but can be checked by hand.
When considering an impartial game $G$, as $\GL = \GR$, we note $G'$ a typical option of $G$.
Observe that for any game $G$ with outcome $\P$, $G+*$ has outcome $\N$.
This is used without reference throughout this paper.
It is also worth noticing that many results in this paper are due to the fact that the only end in the sets of games considered is $0$, that is they all have the same subset of ends.

The paper is organized as follows.
In Section~\ref{sec:out}, we consider tractability to mis\`ere convention of some normal play games properties, giving more evidence that mis\`ere play is in general harder than normal play.
In Section~\ref{sec:bin}, we consider comparison modulo $\ab$, and prove that in infinitely many (non-trivial) cases, it is the same as comparison modulo $\DD$.
Finally, in Section~\ref{sec:imp}, we look at impartial games modulo $\DD$, and prove that comparison modulo $\II$ is the same as comparison modulo $\DD$ when restricted to this particular case.

\section{Comparison between normal and mis\`ere play}
\label{sec:out}

In normal play, dicot games are called all-small, because they are infinitesimal, that is for any positive number $a$ and dicot game $x$, we have $-a \leqslant^+ x \leqslant^+ a$, which also implies $o^+(a+x) = \L$.
In mis\`ere play, this is no longer the case.
In particular, any pair made of a negative number and a positive number is incomparable~\cite{deadending}, which prevents any game to be in the interval between them.
Nevertheless, it is still natural to ask if there is a game $G$ such that for any dicot game $X$, we have $o^-(G+X) = \L$.
Siegel~\cite{canonical} gave all the tools to answer this question by defining the adjoint of a game and giving some of its properties.

\begin{dfn}[Siegel \cite{canonical}]
{\rm
The adjoint of $G$, denoted $G^o$, is given by
$$G^o = 
\begin{cases}
* & \text{ if } G = 0\,, \\
\{(\GR)^o|0\} & \text{ if } G \neq 0 \text{ and } G \text{ is a Left end,} \\
\{0|(\GL)^o\} & \text{ if } G \neq 0 \text{ and } G \text{ is a Right end,} \\
\{(\GR)^o|(\GL)^o\} & \text{ otherwise.}
\end{cases}
$$
where $(\GR)^o$ denotes the set of adjoints of elements of $\GR$.
}
\end{dfn}

\begin{rem}[Dorbec et al. \cite{dicotaxo}]
The adjoint of any game is a dicot game.
\end{rem}

\begin{prop}[Siegel \cite{canonical}]
\label{prop:adjointsum}
For any game $G$, $G+G^o$ is a mis\`ere $\P$-position.
\end{prop}

Thus, for any game $G$, we can find a dicot game, namely $G^o$, such that $o^-(G+G^o) = \P$.
From this, we can naturally find for any game $G$ some dicot games to sum $G$ with and match any outcome.

\begin{cor}
For any game $G$, we have:
\begin{enumerate}[(i)]
\item $o^-(G+G^o) = \P$,
\item $o^-(G+\{G^o|G^o\}) = \N$,
\item $o^-(G+\{G^o,(\GR)^o|(\GL)^o\}) = \L$,
\item $o^-(G+\{(\GR)^o|G^o,(\GL)^o\}) = \R$.
\end{enumerate}
where $(\GL)^o$ (resp. $(\GR)^o$) is replaced by $0$ when $G$ is a Left (resp. Right) end.
\end{cor}

The natural following question is whether we can look at a smaller set $\UU \subset \DD$ to find a game $G$ such that for any game $X$ in $\UU$, we have $o^-(G+X) = \L$.
Among such sets, we answer that problem for the universe of impartial games $\II$ and the set of binary dicot games $\ab$.


We first look at impartial games and define the game $I = \{*|\{*|0\}\}$.
Note that $I$ is a binary dicot game.

\begin{theo}
\label{th:IL}
For any impartial game $X$, we have $o^-(I+X) = \L$.
\end{theo}

\begin{proof}
Let $X$ be an impartial game.
We give a winning strategy for Left on $I + X$:
as long as Right has not played two moves in the $I$ component and there is a move in the $X$ component leaving that component as an $\N$-position, Left plays such a move.
If Right plays two moves in the $I$ component, we know that Left played the last move in the $X$ component, leaving it as an $\N$-position $Y$.
Hence the resulting game is $0+Y$, which Left wins playing first a priori.
Otherwise, at some point there is no move in the $X$ component leaving that component as an $\N$-position, and Left moves the $I$ component to $*$.
If Right moves $*$ to $0$, either there is no move in the $X$ component and Left wins immediately, or she plays any move in that component, resulting in a $\P$-position, which she wins playing second.
If Right moves in the $X$ component, he leaves that component as a $\P$-position $Y$, and Left can move $*$ to $0$.
Hence the resulting game is $0+Y$, which Left wins playing second a priori.
\end{proof}

It might seem surprising that such a simple game, binary dicot with birthday $3$, might overtake any impartial game, but we remind the reader that for any dicot game $X$, we have $o^+(1+X) = \L$, and $1$ seems way simpler than $I$ in our opinion, and dicot games have a richer structure than impartial games.


We now look at binary dicot games and define a family $(B_i)_{i \in \NN}$ as follows:
\begin{itemize}
\item $B_0 = \{0|*\}$,
\item $B_{i+1} = \big\{\{\{0|B_i\}|\{0|B_i\}\}\big|\{0|B_i\}\big\}$.
\end{itemize}
Observe that we can recursively verify that $B_i$ is binary dicot for any $i$.

This family serves here as a counterexample to our prior interrogation, as shown in the following theorem.

\begin{theo}
\label{th:GGiR}
For any game $G$ and any $i$ with $i \geqslant \birth(G)$, we have $o^-(G + B_i) = \R$.
\end{theo}

\begin{proof}
We prove the result by induction on $i$ and $\birth(G)$.
If $i=0$, then $G=0$ and $o^-(0+B_0)=\R$.

Assume now $i \geqslant 1$.
Assume first Right starts the game.
If he has an available move $G^R$ in $G$, he can play it and win by induction as $\birth(G^R) < \birth(G)$.
Otherwise, he plays to $G + \{0|B_{i-1}\}$.
From there, Left can play to $G + 0$ or possibly to some $G^L + \{0|B_{i-1}\}$.
In the first case, Right wins as he has no move available.
In the second case, he can play to $G^L + B_{i-1}$ where he wins by induction as $\birth(G^L) \leqslant \birth(G) - 1 \leqslant i-1$.

Assume now Left starts the game.
If she plays to some $G^L + B_i$, then Right wins by induction as $\birth(G^L) < \birth(G)$.
Otherwise, she plays to $G + \{\{0|B_{i-1}\}|\{0|B_{i-1}\}\}$.
Then Right plays in the $G$ component as long as it is possible and Left is playing is the $G$ component.
If Right cannot play in the $G$ component, the proof is similar to the proof of Right winning when he starts in $G+B_i$ since playing in the $G$ component can only decrease its birthday.
If Left plays to some $G' + \{0|B_{i-1}\}$, Right answers to $G' + B_{i-1}$ where he wins by induction since $\birth(G') \leqslant \birth(G) - 1 \leqslant i-1$ as Right played at least once in the $G$ component.
\end{proof}

This leads to the following corollary.

\begin{cor}
For any game $G$ and any $i$ with $i \geqslant \birth(G)$, we have:
\begin{enumerate}[(i)]
\item $o^-(G+B_i) = \R$,
\item $o^-(G+\overline{B_i}) = \L$,
\item $o^-(G+\{\overline{B_i}|B_i\}) = \N$.
\end{enumerate}
\end{cor}

Unfortunately, we cannot hope for a family $(H_i)_{i \in \NN}$ such that for any game $G$ and any $i$ with $i \geqslant \birth(G)$, we have $o^-(G+H_i)=\P$, as this would mean $o^-(0+H_1) = \P$ which implies $o^-(*+H_1) = \N$.
Nevertheless, it might be possible to construct for any game $G$ a binary dicot game $H_G$ such that $o^-(G+H_G) = \P$, which we leave as an open question.

\begin{question}
Does there exist a game $G$ such that for any binary dicot game $X$, we have \mbox{$o^-(G+X) \neq \P$}?
\end{question}

In mis\`ere play, having a game $G$ greater than any element of a set of games closed under conjugates is not equivalent to having the sum of each game of this set with $G$ be an $\L$-position.
Hence we may still wonder if such a game $G$ exists for the sets we considered earlier.

For dicot games, we may use the adjoint again and prove that the dicot game $(G^o)^o + *$ and $G$ are incomparable for any game $G$, as their sums with $G^o$ have respective outcomes $\N$ and $\P$.
However, we propose here a proof that no such game exists for both binary dicot games and impartial games, which implies the result on dicot games.

We define a family $(s_i)_{i \in \NN}$ of games as follows:
\begin{itemize}
\item $s_0 = 0$,
\item $s_{i+1} = \{s_i|s_i\}$.
\end{itemize}
Observe that we can recursively verify that $s_i$ is both impartial and binary dicot for any $i$, and $s_i$ can be seen as the sum of $i$ games, each of them being $*$.
Actually, we can also recursively verify that any game that is both impartial and binary is of the form $s_i$ for some integer $i$.

Recall that the set of natural integers is recursively defined as $k+1 = \{k|\cdot\}$.

\begin{theo}
For any game $G$ and any integer $i$ with $i \geqslant \birth(G)+1$, $G$ and $s_i$ are incomparable.
\end{theo}

\begin{proof}
Let $G$ be a game and $i$ an integer such that $i \geqslant \birth(G)+1$.

Consider the game $G+i$.
As Left would need at least $i$ moves to get rid of the $i$ component, where Right has no move, Right can win by only playing in $G$ where he cannot play more than $\birth(G)$ moves.
Hence we have $o^-(G+i)=\R$.
Now consider the game $s_i+i$.
Playing first, Left can choose to only play in the $i$ component while Right has no choice but to play in the $s_i$ component.
As both games would last $i$ moves and Left started, Right will play the last move and lose.
Hence we have $o^-(s_i+i)\geqslant\N$.

A similar reasonment would prove that $o^-(G+\overline{i})=\L$ and $o^-(s_i+\overline{i})\leqslant\N$, which concludes the proof.
\end{proof}

Again, it is interesting to see how simple the families $(s_i)_{i \in \NN}$ and $\NN$ are, which emphasizes the complexity of mis\`ere play: in normal play, any $s_i$ is equivalent either to $0$ or to $*$, depending on the parity of $i$; in mis\`ere play, we just proved that they are pairwise incomparable.

It is worth noticing here that the games used to distinguish $G$ and $s_i$ are not dicot, a fortiori neither impartial nor binary dicot.
Hence we might consider the question modulo the universe of dicot games.
For dicot games, the answer is still negative, as the game used to prove that $G$ and $(G^o)^o+*$ are incomparable, namely $G^o$, is dicot, but it might be possible to find a game greater than all impartial games or all binary dicot games modulo the universe of dicot games.
In particular, in the case of binary impartial games, their intersection, such a game exists, and even such a dicot game.

We define a game $S = \big\{0,*\big|\{0,*|0,*\}\big\}$.
Note that $S$ is dicot.

\begin{prop}
For any binary impartial game $G$, we have $S \geqslant^-_\DD G$.
\end{prop}

\begin{proof}
Let $G$ be a binary impartial game.
As mentioned earlier, a binary impartial game is of the form $s_i$ for some integer $i$.
Modulo $\DD$, $s_i$ is equivalent to either $0$ or $*$~\cite{mpg}.
Hence we can consider $G$ to be either $0$ or $*$.

First assume $G=0$ and consider $X$ a dicot game such that Left has a winning strategy on $0+X$ playing first (respectively second).
On $S + X$, Left can follow the same strategy, until either Right plays on the $S$ component or she has no move available in the $X$ component.
In the first case, she can answer in the $S$ component by moving $\{0,*|0,*\}$ to $0$ and resume her winning strategy.
In the second case, it means the $X$ component has been reduced to $0$ and she wins by moving from $S$ to $*$.
Therefore $S \geqslant^-_\DD 0$.

Now assume $G=*$ and consider $X$ a dicot game such that Left has a winning strategy on $*+X$ playing first (respectively second).
On $S + X$, Left can follow the same strategy, unless the strategy recommends that she plays in the $*$ component or Right eventually plays in the $S$ component.
In the first case, the move recommended by the strategy is from $*$ to $0$, hence moving from $S$ to $0$ is also a winning move.
In the second case, she can answer in the $S$ component by moving $\{0,*|0,*\}$ to $*$ and resume her winning strategy.
Therefore $S \geqslant^-_\DD *$.
\end{proof}

As $S$ is dicot, it could be interesting to find an impartial game or a binary game sharing that same property.
Unfortunately, as an impartial game can only have outcome $\P$ or $\N$, no impartial game can be both greater than $0$ and greater than $*$ modulo any set containing $0$.
Moreover, we will see in the next section that any binary game greater than $0$ modulo binary dicot games has outcome $\N$, and as such is also incomparable to $*$ modulo any set containing $0$.

As impartial games seem to have a more predictable behaviour (in particular we have $I\geqslant^-_\II X$ for any impartial game $X$), we highlight the following question.

\begin{question}
Does there exist a game $G$ such that for any impartial game $X$, we have $G \geqslant^-_\DD X$?
\end{question}

In the case the answer is positive, it would also be interesting to find such games $G$ being dicot, as we know that no impartial game would have that property.

\section{Comparison modulo binary dicot games}
\label{sec:bin}

In this section, we focus on binary games, dicot games and their intersection.

First, we prove a useful result on the mis\`ere outcome of the adjoint of any binary game.

\begin{lem}
\label{lem:adjout}
Let $G$ be a binary game.
Then the mis\`ere outcome of $G^o$ is totally determined by the mis\`ere outcome of $G$, namely:
\begin{enumerate}[(i)]
\item if $o^-(G) = \L$, then $o^-(G^o) = \L$,
\item if $o^-(G) = \R$, then $o^-(G^o) = \R$,
\item if $o^-(G) = \N$, then $o^-(G^o) = \P$,
\item if $o^-(G) = \P$, then $o^-(G^o) = \N$.
\end{enumerate}
\end{lem}

\begin{proof}
We prove the result by induction on $G$.
If $G$ is a Left end, then $o^-(G) \geqslant \N$, and as $\Rset{G^o} = \{0\}$, we have $o^-(G^o) \geqslant \P$.
Likewise, if $G$ is a Right end, $o^-(G) \leqslant \N$ and $o^-(G^o) \leqslant \P$.
Assume now $G$ is neither a Left end nor a Right end.
Assume first $o^-(G) \geqslant \N$, that is Left's only move from $G$ is to a position with outcome $\L$ or $\P$.
Then Right's only move from $G^o$ is to a position with outcome $\L$ or $\N$ by induction, and $o^-(G^o) \geqslant \P$.
Similarly, if $o^-(G) \leqslant \N$, then $o^-(G^o) \leqslant \P$.
Assume now $o^-(G) \geqslant \P$, that is Right's only move from $G$ is to a position with outcome $\L$ or $\N$.
Then Left's only move from $G^o$ is to a position with outcome $\L$ or $\P$ by induction, and $o^-(G^o)\geqslant \N$.
Similarly, if $o^-(G) \leqslant \P$, then $o^-(G^o) \leqslant \N$, which concludes the proof.
\end{proof}

This only works with binary games as, for example, $o^-(\{0,*|0\}) = \L$ and $o^-(\{0,*|0\}^o)=\N$.
The argument is that eventhough Left's winning move to $*$ creates a losing move for Right to $*^o$, Left's other move to $0$ is losing and thus creates a winning move to $0^o = *$ for Right.

We can make the following remark about the adjoint of binary games.

\begin{rem}
The adjoint of any binary game is a binary dicot game.
\end{rem}

Using Lemma~\ref{lem:adjout}, we can give the outcome of any binary game greater than or equivalent to $0$ modulo $\ab$.

\begin{prop}
Let $G$ be a binary game such that $G \geqslant^-_\ab 0$.
Then $o^-(G) = \N$.
\end{prop}

\begin{proof}
As $o^-(0) = \N$, we must have $o^-(G) \geqslant \N$.
Assume $o^-(G) = \L$.
Then $o^-(G^o) = \L$.
But $o^-(G+G^o) = \P < \L = o^-(0+G^o)$, contradicting the fact that $G \geqslant^-_\ab 0$.
Hence $o^-(G) = \N$.
\end{proof}

This leads to the following corollary.

\begin{cor}
\label{cor:geq->eq}
Let $G$ be a binary game such that $G \geqslant^-_\ab 0$, $G^{LR}$ exists and $G^{LR} = 0$.
Then $G \equiv^-_\ab 0$.
\end{cor}

\begin{proof}
Let $X$ be a binary dicot game such that Right has a winning strategy on $X$ playing first (respectively second).
On $G + X$, Right can follow the same strategy, until either Left plays on the $G$ component or he has no move available in the $X$ component.
In the first case, he can answer in the $G$ component by moving from $G^L$ to $G^{LR} = 0$.
In the second case, as $G \geqslant^-_\ab 0$, we have $o^-(G) = \N$, and as the $X$ component has been reduced to $0$, Right wins playing first in $G$.
Hence $G \leqslant^-_\ab 0$.

As we assumed $G \geqslant^-_\ab 0$, we have $G \equiv^-_\ab 0$.
\end{proof}

The reader familiar with canonical forms of games would have recognized that $G^L$ is $\ab$-reversible through $0$.
In particular, Corollary~\ref{cor:geq->eq} can be rephrased as {\it A binary game that has a $\ab$-reversible option through $0$ is $\ab$-equivalent to $0$}.
A slight modification of results presented in~\cite{dicotaxo, canonical} (see Lemmas~\ref{34} and~\ref{54} at the end of this section) would lead to a canonical form for binary dicot games modulo $\ab$, but as we show at the end of this section, it would be the same form as modulo $\DD$, hence we only mention its existence.

Despite Corollary~\ref{cor:geq->eq}, which could make us believe that many binary dicot games greater than or equal to $0$ modulo $\ab$ are actually equivalent to $0$, there exist some binary dicot games that are strictly greater than $0$ modulo $\ab$.
We here give an example and define $Z = \big\{\{*|\{*|0\}\}\big|*\big\}$ and $G_a = \{0|*\}$.
Observe that $Z$ is a mis\`ere $\N$-position and $G_a$ is a mis\`ere $\R$-position.

\begin{prop}
\label{prop:Z>0}
We have $Z >^-_\ab 0$.
\end{prop}

\begin{proof}
Let $X$ be a binary dicot game such that Left has a winning strategy on $X$ playing first (respectively second).
On $Z + X$, Left can follow the same strategy, until either Right plays on the $Z$ component or she has no move available in the $X$ component.
In the first case, she can answer in the $Z$ component by moving from $*$ to $0$.
In the second case, as $X$ is dicot, it means the players have reduced $X$ to $0$, and as $Z$ is an $\N$-position, Left wins playing first a priori.
Hence $Z \geqslant^-_\ab 0$. \\
To see that the inequality is strict, one needs only see that $o^-(0+G_a) = \R$ while $o^-(Z+G_a) = \L$.
\end{proof}

We now want to compare comparison modulo $\ab$ with comparison modulo $\DD$.
The following results lead to Theorems~\ref{th:B=>D0}, \ref{th:B=>D} and~\ref{th:B=>Db}, which we consider the most interesting results of this paper, together with Theorem~\ref{th:I=>D}.

We first focus on the game $0$ and give a sufficient condition for a game to be greater than or equal to $0$ modulo $\ab$.

\begin{lem}
\label{lem:1step0}
Let $G$ be a game with mis\`ere outcome $\N$ or $\L$ such that for any Right option $G^R$ of $G$, there exists a Left option $G^{RL}$ of $G^R$ with $G^{RL} \geqslant^-_\ab 0$.
Then $G \geqslant^-_\ab 0$.
\end{lem}

\begin{proof}
Let $X$ be a binary dicot game such that Left has a winning strategy on $X$ playing first (respectively second).
On $G+X$, Left can follow the same strategy, until either Right plays on the $G$ component or she has no move available in the $X$ component.
In the first case, she can answer in the $G$ component by moving from $G^R$ to some $G^{RL}$ with $G^{RL} \geqslant^-_\ab 0$.
In the second case, as $G$ has mis\`ere outcome $\N$ or $\L$ and $X$ has been reduced to $0$, Left wins playing first in $G+0$.
Hence $G \geqslant^-_\ab 0$.
\end{proof}

Using Lemma~\ref{lem:1step0}, we can give a characterisation of games greater than or equivalent to $0$ modulo $\ab$.

\begin{lem}
\label{lem:carac0}
Let $G$ be a game and $i$ an integer with $i \geqslant \birth(G)$.
Then we have $G \geqslant^-_\ab 0$ if and only if Left has a winning strategy on $G+\{B_i|0\}$ playing second.
\end{lem}

\begin{proof}
As $\{B_i|0\}$ is a binary dicot game with outcome $\P$, if $G \geqslant^-_\ab 0$, then Left has a winning strategy on $G+\{B_i|0\}$ playing second.

Assume now Left has a winning strategy on $G + \{B_i|0\}$ playing second.
We prove the result by induction on $G$.
If $G = 0$, then $G \geqslant^-_\ab 0$.
Assume now $G \neq 0$.
As Right can play from $G + \{B_i|0\}$ to $G + 0$, $o^-(G)$ has to be $\N$ or $\L$.
Assume Right plays from $G + \{B_i|0\}$ to some $G^R + \{B_i|0\}$.
Then Left has a winning answer, which cannot be to $G^R + B_i$ since Theorem~\ref{th:GGiR} states its outcome is $\R$.
Then there exists a Left option $G^{RL}$ of $G^R$ such that Left has a winning strategy on $G^{RL}+\{B_i|0\}$ playing second.
By induction, we have $G^{RL} \geqslant^-_\ab 0$.
Hence for any Right option $G^R$ of $G$, there exists a Left option $G^{RL}$ of $G^R$ with $G^{RL} \geqslant^-_\ab 0$.
Then by Lemma~\ref{lem:1step0}, $G \geqslant^-_\ab 0$.
\end{proof}

The proof of Lemma~\ref{lem:carac0} has for immediate consequence the converse of Lemma~\ref{lem:1step0}.

\begin{cor}
\label{cor:HI0}
Let $G$ be a game such that $G \geqslant^-_\ab 0$.
Then for any Right option $G^R$ of $G$, there exists a Left option $G^{RL}$ of $G^R$ such that $G^{RL} \geqslant^-_\ab 0$.
\end{cor}

We now have the tools needed to state Theorem~\ref{th:B=>D0}.

\begin{theo}
\label{th:B=>D0}
Let $G$ be a game.
Then we have $G \geqslant^-_\DD 0$ if and only if we have $G \geqslant^-_\ab 0$.
\end{theo}

\begin{proof}
As $\ab$ is a subset of $\DD$, we naturally have $G \geqslant^-_\ab 0$ whenever we have $G \geqslant^-_\DD 0$.

Assume now $G \geqslant^-_\ab 0$.
We prove the result by induction on $G$.
If $G = 0$, we have both $G \geqslant^-_\ab 0$ and $G \geqslant^-_\DD 0$.
Now assume $G \neq 0$.
Let $X$ be a dicot game such that Left has a winning strategy on $X$ playing first (respectively second).
On $G+X$, Left can follow the same strategy, until either Right plays on the $G$ component or she has no move available in the $X$ component.
In the first case, Corollary~\ref{cor:HI0} ensures she can answer in the $G$ component by moving from $G^R$ to some $G^{RL}$ with $G^{RL} \geqslant^-_\ab 0$.
By induction, we have $G^{RL} \geqslant^-_\DD 0$, hence Left wins the game a priori.
In the second case, as we have $G \geqslant^-_\ab 0$, $G$ has mis\`ere outcome $\N$ or $\L$, and $X$ has been reduced to $0$, so Left wins playing first in $G+0$.
Hence $G \geqslant^-_\DD 0$.
\end{proof}

We would like here to emphasize the fact that in Theorem~\ref{th:B=>D0}, $G$ ranges in the universe of all games.
In particular, as $G \geqslant^-_\DD H$ implies $G \geqslant^+ H$~\cite{dicotaxo}, we get that only normal $\P$-positions might be equivalent to $0$ modulo $\ab$.

In the following, we somehow extend this result by replacing $0$ by a larger set of games.
Unfortunately, to get there, we also reduce the set in which we choose $G$.
First, we give the following definition.

\begin{dfn}
{\rm
For a game $G$ and an integer $i$, we note $\widetilde{G^{}}^i$ the game given by
$$\widetilde{G^{}}^i = 
\begin{cases}
\{B_i|0\} & \text{ if } G = 0\,, \\
\{\widetilde{(\GR)}^i|0\} & \text{ if } G \neq 0 \text{ and } G \text{ is a Left end,} \\
\{B_i|\widetilde{(\GL)}^i\} & \text{ if } G \neq 0 \text{ and } G \text{ is a Right end,} \\
\{\widetilde{(\GR)}^i|\widetilde{(\GL)}^i\} & \text{ otherwise.}
\end{cases}
$$
}
\end{dfn}

Note that this definition looks quite similar to the definition of the adjoint, and the $\widetilde{G^{}}^i$ games actually share several properties with $G^o$, that we state here, the proofs being similar to the proofs of the similar properties for the adjoint.

\begin{rem}
If $G$ is a binary game, then $\widetilde{G^{}}^i$ is a binary dicot game.
\end{rem}

\begin{lem}
\label{lem:adjiout}
Let $G$ be a binary game.
Then the mis\`ere outcome of $\widetilde{G^{}}^i$ is totally determined by the mis\`ere outcome of $G$, namely:
\begin{enumerate}[(i)]
\item if $o^-(G) = \L$, then $o^-(\widetilde{G^{}}^i) = \L$,
\item if $o^-(G) = \R$, then $o^-(\widetilde{G^{}}^i) = \R$,
\item if $o^-(G) = \N$, then $o^-(\widetilde{G^{}}^i) = \P$,
\item if $o^-(G) = \P$, then $o^-(\widetilde{G^{}}^i) = \N$.
\end{enumerate}
\end{lem}

\begin{prop}
For any game $G$ and any integer $i$ such that $i \geqslant \birth(G)$, $G+\widetilde{G^{}}^i$ has mis\`ere outcome $\P$.
\end{prop}

We now prove some intermediate lemmas to get to the proof of Theorem~\ref{th:B=>D}.
They are similar to the lemmas we proved to get to Theorem~\ref{th:B=>D0}, with other sets of games.

\begin{lem}
\label{lem:1step}
Let $G$ and $H$ be games such that 
\begin{enumerate}[1.]
\item if $G$ is a Right end, then the mis\`ere outcome of $H$ is either $\N$ or $\R$.
\item for any Right option $G^R$ of $G$, there exists a Right option $H^R$ with $G^R \geqslant^-_\ab H^R$ or a Left option $G^{RL}$ of $G^R$ with $G^{RL} \geqslant^-_\ab H$,
\item if $H$ is a Left end, then the mis\`ere outcome of $G$ is either $\N$ or $\L$.
\item for every Left option $H^L$ of $H$, there exists a Left option $G^L$ of $G$ with $G^L \geqslant^-_\ab H^L$ or a Right option $H^{LR}$ of $H^L$ with $G \geqslant^-_\ab H^{LR}$.
\end{enumerate}
Then $G \geqslant^-_\ab H$.
\end{lem}

\begin{proof}
Let $X$ be a binary dicot game such that Left has a winning strategy on $H+X$ playing first (respectively second).
On $G+X$, Left can follow the same strategy, until either Right plays on the $G$ component from $G$ to some $G^R$, the strategy recommends that she plays on the $H$ component from $H$ to some $H^L$, or the two players reduces $X$ to $0$.
In the first case, she can either consider Right played from $H$ to some $H^R$ with $G^R \geqslant H^R$, or answer in the $G$ component by moving from $G^R$ to some $G^{RL}$ with $G^{RL} \geqslant^-_\ab H$.
In the second case, she can either play in the $G$ component from $G$ to some $G^L$ with $G^L \geqslant^-_\ab H^L$ or consider she moved to $H^L$ and Right answered to some $H^{LR}$ with $G \geqslant^-_\ab H^{LR}$.
In the third case, if it is Right's turn to play, $H$ has outcome $\P$ or $\L$ so $G$ is not a Right end and the same argument as in the first case ensures Left wins.
Assume then it is Left's turn to play.
If $H$ is not a Left end, the same argument as in the second case ensures Left wins.
Otherwise, as the mis\`ere outcome of $G$ is either $\N$ or $\L$, Left wins a priori.
Hence $G \geqslant^-_\ab H$.
\end{proof}

Using Lemma~\ref{lem:1step}, we can give a characterisation of games greater than or equivalent to a binary game $H$ modulo $\ab$, given that no follower of $H$ has outcome $\L$.

\begin{lem}
\label{lem:carac}
Let $G$ be a dicot game, $i$ an integer with $i \geqslant {\rm max}(\birth(G),\birth(H))$ and $H$ a binary game such that no follower of $H$ has outcome $\L$.
Then we have $G \geqslant^-_\ab H$ if and only if Left has a winning strategy on $G+\widetilde{H^{}}^i$ playing second.
\end{lem}

\begin{proof}
As $\widetilde{H^{}}^i$ is a binary dicot game such that $H+\widetilde{H^{}}^i$ has mis\`ere outcome $\P$, if $G \geqslant^-_\ab H$, then Left has a winning strategy on $G+\widetilde{H^{}}^i$ playing second.

Assume now Left has a winning strategy on $G + \widetilde{H^{}}^i$ playing second.
Left's move to any follower of $G$ summed with $B_i$ is always losing, which is why we never consider it among potential winning moves in this proof.
We prove the result by induction on $G$ and $H$.
If $H$ is a Left end, as Right can move from $G + \widetilde{H^{}}^i$ to $G$, the mis\`ere outcome of $G$ is either $\N$ or $\L$.
If $G$ is a Right end, as $G$ is dicot, we have $G=0$, hence as Left can win $G+\widetilde{H^{}}^i = \widetilde{H^{}}^i$ playing second, $\widetilde{H^{}}^i$ has outcome $\P$ or $\L$, which means $H$ has outcome $\N$ or $\L$ by Lemma~\ref{lem:adjiout}.
As we assumed $H$ does not have outcome $\L$, $H$ has outcome $\N$.

Assume first Right plays from $G + \widetilde{H^{}}^i$ to some $G^R + \widetilde{H^{}}^i$.
Then Left has a winning answer, which is either to some $G^{RL} + \widetilde{H^{}}^i$ or to some $G^R + \widetilde{H^R}^i$.
In the first case, there exists a Left option $G^{RL}$ of $G^R$ such that Left has a winning strategy on $G^{RL}+\widetilde{H^{}}^i$ playing second.
By induction, we have $G^{RL} \geqslant^-_\ab H$.
In the second case, there exists a Right option $H^R$ such that Left has a winning strategy on $G^R + \widetilde{H^R}^i$ playing second.
By induction, we have $G^R \geqslant^-_\ab H^R$.
Hence for any Right option $G^R$ of $G$, there exists a Right option $H^R$ of $H$ such that $G^R \geqslant^-_\ab H^R$ or a Left option $G^{RL}$ of $G^R$ with $G^{RL} \geqslant^-_\ab H$.

Assume now Right plays from $G + \widetilde{H^{}}^i$ to some $G + \widetilde{H^L}^i$.
Then Left has a winning answer, which is either to $G^L + \widetilde{H^L}^i$ or to $G + \widetilde{H^{LR}}^i$.
With a reasonment similar to the previous paragraph, we get that for any Left option $H^L$ of $H$, there exists a Left option $G^L$ of $G$ with $G^L \geqslant^-_\ab H^L$ or a Right option $H^{LR}$ of $H^L$ with $G \geqslant^-_\ab H^{LR}$.

Then by Lemma~\ref{lem:1step}, $G \geqslant^-_\ab H$.
\end{proof}

Here, we added the extra condition that $G$ needs to be dicot.
The problem is we cannot deal with Right ends which are not $0$, and as the proof is by induction we only consider dicot games.
To see that the result becomes false when you remove the dicot condition, consider $G=1$, $i=1$ and $H=*$.
Then Left has a winning strategy playing second in $1 + \widetilde{*^{}}^1$, but $1 \ngeqslant^-_\ab *$ as $o^-(0+1) = \R$ and $o^-(0+*) = \P$.

We also added the condition that $H$ needs to be binary, and again, this condition cannot be removed: 
consider $G=0$, $i=0$ and $H=\big\{\{0|0,*\}\big|0\big\}$.
Then as $o^-(0) = \N$, $o^-(H) = \P$ and $o^-(\widetilde{H^{}}^0) = \L$, we have $0$ and $H$ incomparable modulo $\ab$ though Left has a winning strategy playing second in $0 + \widetilde{H^{}}^0$ and $0 \geqslant \birth(0)$.

The third condition we added is that $H$ has no follower with outcome $\L$, which again cannot be removed:
consider $G=0$, $i=0$ and $H = Z = \big\{\{*|\{*|0\}\}\big|*\big\}$.
We saw in Proposition~\ref{prop:Z>0} that $Z >^-_\ab 0$, hence we cannot have $0 \geqslant^-_\ab Z$, whereas Left has a winning strategy on $0 + \widetilde{Z^{}}^0$ playing second as $\widetilde{Z^{}}^0$ has outcome $\P$.

The proof of Lemma~\ref{lem:carac} has for immediate consequence the converse of Lemma~\ref{lem:1step}, with the additional hypothesis that $G$ is dicot and $H$ is binary with no follower having outcome $\L$.

\begin{cor}
\label{cor:HI}
Let $G$ be a dicot game and $H$ a binary game such that $G \geqslant^-_\ab H$ and $H$ has no follower with outcome $\L$.
Then
\begin{enumerate}[1.]
\item if $G$ is a Right end, then the mis\`ere outcome of $H$ is either $\N$ or $\R$.
\item for any Right option $G^R$ of $G$, there exists a Right option $H^R$ with $G^R \geqslant^-_\ab H^R$ or a Left option $G^{RL}$ of $G^R$ with $G^{RL} \geqslant^-_\ab H$,
\item if $H$ is a Left end, then the mis\`ere outcome of $G$ is either $\N$ or $\L$.
\item for every Left option $H^L$ of $H$, there exists a Left option $G^L$ of $G$ with $G^L \geqslant^-_\ab H^L$ or a Right option $H^{LR}$ of $H^L$ with $G \geqslant^-_\ab H^{LR}$.
\end{enumerate}
\end{cor}

We are now in position to state Theorem~\ref{th:B=>D}.

\begin{theo}
\label{th:B=>D}
Let $G$ be a dicot game and $H$ a binary game with no follower having outcome $\L$.
Then we have $G \geqslant^-_\DD H$ if and only if we have $G \geqslant^-_\ab H$.
\end{theo}

\begin{proof}
As $\ab$ is a subset of $\DD$, we naturally have $G \geqslant^-_\ab H$ whenever we have $G \geqslant^-_\DD H$.

Assume now $G \geqslant^-_\ab H$.
We prove the result by induction on $G$ and $H$.
Let $X$ be a dicot game such that Left has a winning strategy on $H+X$ playing first (respectively second).
On $G+X$, Left can follow the same strategy, until either Right plays on the $G$ component from $G$ to some $G^R$ or the strategy recommends that she plays on the $H$ component from $H$ to some $H^L$, or the players reduce the $X$ component to $0$.
In the first case, Corollary~\ref{cor:HI} ensures she can answer in the $G$ component by moving from $G^R$ to some $G^{RL}$ with $G^{RL} \geqslant^-_\ab H$ or consider Right moved in the $H$ component from $H$ to some $H^R$ with $G^R \geqslant^-_\ab H^R$.
By induction, we have $G^{RL} \geqslant^-_\DD H$ or $G^R \geqslant^-_\DD H^R$, hence Left wins the game a priori.
In the second case, Corollary~\ref{cor:HI} ensures she can play to some $G^L$ with $G^L \geqslant^-_\ab H^L$ or consider Right moved in the $H$ component from $H^L$ to some $H^{LR}$ with $G \geqslant^-_\ab H^{LR}$.
By induction, we have $G^L \geqslant^-_\DD H^L$ or $G \geqslant^-_\DD H^{LR}$, hence Left wins the game a priori.
In the third case, if it is Right's turn to play, $H$ has outcome $\P$ or $\L$ so $G$ is not a Right end by Corollary~\ref{cor:HI} and the same argument as in the first case ensures Left wins.
Assume then it is Left's turn to play.
If $H$ is not a Left end, the same argument as in the second case ensures Left wins.
Otherwise, as the mis\`ere outcome of $G$ is either $\N$ or $\L$ by Corollary~\ref{cor:HI}, Left wins a priori.
Hence $G \geqslant^-_\DD H$.
\end{proof}

Though Lemma~\ref{lem:carac} cannot be extended by choosing $G$ in all games, nor by choosing $H$ to all dicot games, we might still hope to extend Theorem~\ref{th:B=>D}.
In particular, it would be really interesting to know whether having $G \geqslant^-_\ab H$ implies $G \geqslant^-_\DD H$ when $G$ and $H$ are both dicot games since equivalence modulo dicot games is mostly used between dicot games, but having a similar result considering all games would still give even more meaning to the set $\ab$.
We here present a proof of a similar result when $G$ and $H$ are both binary, using a slightly different method.
We need some results from~\cite{dicotaxo} and~\cite{canonical}, that we recall here, together with some counterparts of other results from these papers adapted to binary dicot games.

We first recall the following proposition.

\begin{prop}[Dorbec et al.~\cite{dicotaxo}]
\label{21}
Let $\mathcal{U}$ be a set of games, $G$ and $H$ two games (not necessarily in $\UU$).
We have $G \geqslant^-_\UU H$ if and only if the following two conditions hold:
\begin{enumerate}[(i)]
\item For all $X \in \mathcal{U}$ with $o^-(H+X) \geqslant \P$, we have $o^-(G+X) \geqslant \P$; and
\item For all $X \in \mathcal{U}$ with $o^-(H+X) \geqslant \N$, we have $o^-(G+X) \geqslant \N$.
\end{enumerate}
\end{prop}

We can now adapt the following lemma to binary dicot games, being careful about the construction staying in $\ab$, in particular by only considering binary games.

\begin{lem}
\label{34}
Let $G$ and $H$ be any binary games.
If $G \ngeqslant^-_\ab H$, then: 
\begin{enumerate}[(a)]
\item There exists some $Y \in \ab$ such that $o^-(G+Y) \leqslant \P$ and $o^-(H+Y) \geqslant \N$; and 
\item There exists some $Z \in \ab$ such that $o^-(G+Z) \leqslant \N$ and $o^-(H+Z) \geqslant \P$.
\end{enumerate}
\end{lem}

\begin{proof}
Negating the condition of Proposition~\ref{21}, we get that (a) or (b) must hold. To prove the lemma, we show that (a) $\Rightarrow$ (b) and (b) $\Rightarrow$ (a).

Consider some $Y \in \ab$ such that $o^-(G+Y) \leqslant \P$ and $o^-(H+Y) \geqslant \N$, and set
$$
Z =
\begin{cases}
\{0|Y\} & \text{ if } H \text{ is a Right end} \\
\{(H^R)^o|Y\} & \text{ otherwise.}
\end{cases}
$$
First note that since $Z$ has both a Left and a Right option, and both these options are binary dicot, $Z$ is also binary dicot.
We now show that $Z$ satisfies $o^-(G+Z) \leqslant \N$ and $o^-(H+Z) \geqslant \P$, as required in (b).
From the game $G+Z$, Right has a winning move to $G+Y$, so $o^-(G+Z) \leqslant \N$.
We now prove that Right has no winning move in the game $H+Z$.
Observe first that $H+Z$ is not a Right end since $Z$ is not.
If Right moves to $H^R+Z$, Left has a winning response to $H^R + (H^R)^o$.
If instead Right moves to $H+Y$ then, since $o^-(H+Y) \geqslant \N$, Left wins a priori.
Therefore $o^-(H+Z) \geqslant \P$, and (a) $\Rightarrow$ (b).

To prove (b) $\Rightarrow$ (a), for a given $Z$ we set $Y = \{Z|0\}$ when $G$ is a Left end and $\{Z|(G^L)^o\}$ when $G$ is not a Left end and prove similarly that Left wins if she plays first on $H+Y$ and loses if she plays first on $G+Y$.
\end{proof}

We now recall the following definition and lemma, that will be useful in the following.

\begin{dfn}[Siegel~\cite{canonical}]
{\rm
Let $G$ and $H$ be any two games and $\UU$ a set of games.
If there exists some $T \in \UU$ such that $o^-(G+T) \leqslant \P \leqslant o^-(H+T)$, we say that $G$ is $\UU$-downlinked to $H$ (by $T$).
In that case, we also say that $H$ is $\UU$-uplinked to $G$ by $T$.
}
\end{dfn}

\begin{lem}[Dorbec et al.~\cite{dicotaxo}]
\label{53}
Let $G$ and $H$ be any two games and $\UU$ be a set of games.
If $G \geqslant^-_{\mathcal{U}} H$, then $G$ is $\mathcal{U}$-downlinked to no $H^L$ and no $G^R$ is $\mathcal{U}$-downlinked to $H$.
\end{lem}

We can now adapt the following lemma to binary dicot games, being careful that the construction stays in $\ab$, again by only considering binary games.

\begin{lem}
\label{54}
Let $G$ and $H$ be any binary games.
$G$ is $\ab$-downlinked to $H$ if and only if no $G^L \geqslant^-_\ab H$ and no $H^R \leqslant^-_\ab G$.
\end{lem}

\begin{proof}
Consider two binary games $G$ and $H$ such that $G$ is $\ab$-downlinked to $H$ by some binary dicot game $T$, i.e. $o^-(G + T) \leqslant \P \leqslant o^-(H + T)$.
Then Left has no winning move from $G+T$, thus $o^-(G^L + T) \leqslant \N$ and similarly $o^-(H^R + T) \geqslant \N$.
Therefore, $T$ witnesses both $G^L \ngeqslant^-_\ab H$ and $G \ngeqslant^-_\ab H^R$.

Conversely, suppose that no $G^L \geqslant^-_\ab H$ and no $H^R  \leqslant^-_\ab G$.
By Lemma~\ref{34}, if $G$ is not a Left end, we can associate to $G^L$ a game $X \in \ab$ such that $o^-(G^L + X) \leqslant \P$ and $o^-(H + X) \geqslant \N$.
Likewise, if $H$ is not a Right end, we associate to $H^R$ a game $Y \in \ab$ such that $o^-(G + Y) \leqslant \N$ and $o^-(H^R + Y) \geqslant \P$.
Let $T$ be the game defined by 
$$
\begin{array}{rcll}

\Lset T & = & \left\{ 

\begin{array}{l}
\{ 0 \} \\
\{ (G^R)^o \} \\
\{ Y \} 
\end{array}
\right.
&

\begin{array}{l}
  \text{ if both $G$ and $H$ are Right ends,} \\
  \text{ if $H$ is a Right end and $G$ is not,} \\
  \text{ otherwise.}
\end{array}  \\

\Rset T & = & \left\{ \begin{array}{l}
\{ 0 \} \\
\{ (H^L)^o \} \\
\{ X \} \\
\end{array}
\right.

&
\begin{array}{l}
 \text{ if both $G$ and $H$ are Left ends,} \\
 \text{ if $G$ is a Left end and $H$ is not,} \\
 \text{ otherwise.}
\end{array} 

 \end{array}
$$

As $T$ has both a Left option and a Right option, and both these options are binary dicot, $T$ is binary dicot.
We claim that $G$ is $\ab$-downlinked to $H$ by $T$.

To show that $o^-(G + T) \leqslant \P$, we just prove that Left loses if she plays first in $G+T$.
Since $T$ has a Left option, $G+T$ is not a Left end. 
If Left moves to $G^L + T$, then by our choice of $X$, Right has a winning response to $G^L + X$.
If Left moves to $G + (G^R)^o$, then Right can respond to $G^R + (G^R)^o$ and win.
If Left moves to $G + Y$, then by our choice of $Y$, $o^-(G + Y) \leqslant \N$ and Right wins a priori.
The only remaining possibility is, when $G$ and $H$ are both Right ends, that Left moves to $G + 0$.
But then Right cannot move and wins.

Now, we show that $o^-(H + T) \geqslant \P$ by proving that Right loses playing first in $H+T$.
If Right moves to $H^R+T$, then Left has a winning response to $H^R + Y$.
If Right moves to $H+(H^L)^o$, then Left wins by playing to $H^L + (H^L)^o$, and if Right moves to $H+X$, then by our choice of $X$, $o^-(H + X) \geqslant \N$ and Left wins a priori.
Finally, the only remaining possibility, when $G$ and $H$ are both Left ends, is that Right moves to $0$.
But then Left cannot answer and wins.
\end{proof}

With this, we can state some converse of Lemma~\ref{lem:1step}, restricted to binary games.

\begin{lem}
\label{lem:binrec}
Let $G$ and $H$ be any binary games.
If $G \geqslant^-_\ab H$, then 
\begin{enumerate}[1.]
\item if $G$ is a Right end, then the mis\`ere outcome of $H$ is either $\N$ or $\R$.
\item for any Right option $G^R$ of $G$, there exists a Right option $H^R$ with $G^R \geqslant^-_\ab H^R$ or a Left option $G^{RL}$ of $G^R$ with $G^{RL} \geqslant^-_\ab H$,
\item if $H$ is a Left end, then the mis\`ere outcome of $G$ is either $\N$ or $\L$.
\item for every Left option $H^L$ of $H$, there exists a Left option $G^L$ of $G$ with $G^L \geqslant^-_\ab H^L$ or a Right option $H^{LR}$ of $H^L$ with $G \geqslant^-_\ab H^{LR}$.
\end{enumerate}
\end{lem}

\begin{proof}
1. and 3. are immediate since otherwise we would have $o^-(G+0) \ngeqslant o^-(H+0)$.

Now consider the Right option $G^R$ of $G$ when it exists.
As we have $G \geqslant^-_\ab H$, by Lemma~\ref{53}, $G^R$ is not $\ab$-downlinked to $H$.
Hence by Lemma~\ref{54}, there exists a Left option $G^{RL}$ of $G^R$ such that $G^{RL} \geqslant^-_\ab H$ or a Right option $H^R$ of $H$ such that $G^R \geqslant^-_\ab H^R$, which is exactly 2. \\
The proof of 4. is similar to the proof of 2.
\end{proof}

We can now state Theorem~\ref{th:B=>Db}.

\begin{theo}
\label{th:B=>Db}
Let $G$ and $H$ be any binary game.
We have $G \geqslant^-_\ab H$ if and only if we have $G \geqslant^-_\DD H$.
\end{theo}

\begin{proof}
As $\ab$ is a subset of $\DD$, we naturally have $G \geqslant^-_\ab H$ whenever we have $G \geqslant^-_\DD H$.

Assume now $G \geqslant^-_\ab H$.
We prove the result by induction on $G$ and $H$.
Let $X$ be a dicot game such that Left has a winning strategy on $H+X$ playing first (respectively second).
On $G+X$, Left can follow the same strategy, until either Right plays on the $G$ component from $G$ to $G^R$ or the strategy recommends that she plays on the $H$ component from $H$ to $H^L$, or the players reduce the $X$ component to $0$.
In the first case, Lemma~\ref{lem:binrec} ensures she can answer in the $G$ component by moving from $G^R$ to $G^{RL}$ with $G^{RL} \geqslant^-_\ab H$ or consider Right moved in the $H$ component from $H$ to $H^R$ with $G^R \geqslant^-_\ab H^R$.
By induction, we have $G^{RL} \geqslant^-_\DD H$ or $G^R \geqslant^-_\DD H^R$, hence Left wins the game a priori.
In the second case, Lemma~\ref{lem:binrec} ensures she can play to $G^L$ with $G^L \geqslant^-_\ab H^L$ or consider Right moved in the $H$ component from $H^L$ to $H^{LR}$ with $G \geqslant^-_\ab H^{LR}$.
By induction, we have $G^L \geqslant^-_\DD H^L$ or $G \geqslant^-_\DD H^{LR}$, hence Left wins the game a priori.
In the third case, if it is Right's turn to play, $H$ has outcome $\P$ or $\L$ so $G$ is not a Right end by Lemma~\ref{lem:binrec} and the same argument as in the first case ensures Left wins.
Assume then it is Left's turn to play.
If $H$ is not a Left end, the same argument as in the second case ensures Left wins.
Otherwise, as the mis\`ere outcome of $G$ is either $\N$ or $\L$ by Lemma~\ref{lem:binrec}, Left wins a priori.
Hence $G \geqslant^-_\DD H$.
\end{proof}

As we announced earlier in this section, Theorem~\ref{th:B=>Db} implies that we cannot reduce a binary dicot game more by considering it modulo binary dicot games rather than modulo all dicot games.
This implies in particular that if the canonical form of a dicot game (modulo $\DD$) is not binary, then this game cannot be equivalent to any binary dicot game modulo $\DD$, as it is easy to verify that the canonical form of a binary dicot game modulo the dicot games (as defined in~\cite{dicotaxo}) is a binary dicot game.
This emphasizes the fact that binary dicot games do not reach all equivalence classes of dicot games modulo $\DD$, as for example there are $1268$ equivalence classes of dicot games with birthday at most $3$ modulo $\DD$\cite{dicotaxo}, and only $26$ binary dicot trees with birthday at most $3$, among them only $13$ are in canonical form.
Nevertheless, the equivalence classes they reach seem to be those that matter more, as the others add nothing when comparing binary games, and in many other cases.

\section{Comparison modulo impartial games}
\label{sec:imp}

In this section, we focus on impartial games and dicot games.

First, we recall some definitions and results about impartial games.

\begin{dfn}[Conway~\cite{onag}]
An option $G'$ of an impartial game $G$ is said to be $\II$-reversible (through $G''$) if $G'' \equiv^-_\II G$ for some option $G''$ of $G'$.
\end{dfn}

\begin{dfn}[Conway~\cite{onag}]
An impartial game $G$ is said to be in impartial canonical form if no follower of $G$ has any $\II$-reversible option.
\end{dfn}

\begin{theo}[Conway~\cite{onag}]
Consider two impartial games $G$ and $H$ in impartial canonical form with $G \equiv^-_\II H$.
Then $G=H$.
\end{theo}

\begin{theo}[Siegel~\cite{cgt}]
\label{th:revsimp}
Consider two impartial games $G$ and $H$ such that every option of $H$ is in impartial canonical form, and some option of $H$ is reversible through $G$.
Then
\begin{enumerate}[(i)]
\item every option of $G$ is an option of $H$,
\item every other option of $H$ is reversible through $G$.
\end{enumerate}
\end{theo}

Another useful observation is the following, using the fact that impartial games are their own conjugates.

\begin{obs}
\label{obs:imp}
Let $G$ and $H$ be two impartial games and $\UU$ a set of games closed by conjugates.
If $G \geqslant^-_\UU H$, then $G \equiv^-_\UU H$.
\end{obs}

We are now in position to state the main result of this section.

\begin{theo}
\label{th:I=>D}
Let $G$ and $H$ be two impartial games such that $G \equiv^-_\II H$.
Then $G\equiv^-_\DD H$.
\end{theo}

\begin{proof}
By induction, we can consider that $G$ and all options of $H$ are in impartial canonical form.
If $H$ is in impartial canonical form, then we have $G=H$, and so $G \equiv^-_\DD H$.
Hence we can assume $H$ is not in impartial canonical form.
This means there is a reversible option $H'$ of $H$ through an option $H''$ (of $H'$) with $H \equiv^-_\II H''$.
As $H'$ is in impartial canonical form, so is $H''$, and $H''=G$.
By Theorem~\ref{th:revsimp}, every option of $G$ is an option of $H$, and every option of $H$ that is not an option of $G$ has $G$ as one of its options.

Now consider a dicot game $X$ such that Left wins $G+X$ playing first (respectively second).
In $H+X$, she can follow the same strategy until either Right plays in $H$, or her strategy recommends a move in $G$, or there is no more any move available in the $X$ component.
In the first case, either he moved the $H$ component to a position $H'$ that is an option of $G$, and she can assume he played that move and resume her strategy, or $G$ is an option of $H'$, so Left can just move the $H$ component to $G$ and resume her strategy.
In the second case, her strategy recommends her to move the $G$ component to some $G'$ that is also an option of $H$, so she can move the $H$ component to $G'$ and resume her strategy.
In the third case, as $G$ and $H$ are $\II$-equivalent, they have the same outcome, hence as Left was winning $G$, she wins $H$ a priori.
Therefore, we have $H \geqslant^-_\DD G$, and so $H \equiv^-_\DD G$ by Observation~\ref{obs:imp}.
\end{proof}

Note that the converse is obviously true since $\II \subset \DD$.

Unfortunately, it is quite unlikely that we can extend this result much more, as we now give several counterexamples to some `extensions' that would have been natural to consider.

The first potential extension we considered is: Do we have $G \geqslant^-_\II H \Rightarrow G \geqslant^-_\DD H$ whenever $G$ is dicot and $H$ is impartial?
Unfortunately, even reducing $H$ to only be $0$ is not enough if we want $G$ to be able to range over all dicot games.

\begin{prop}
\begin{enumerate}[a)]
\item $I \geqslant^-_\II 0$,
\item $I \ngeqslant^-_\DD 0$.
\end{enumerate}
\end{prop}

\begin{proof}
Theorem~\ref{th:IL} tells us that whichever impartial game you add to $I$, the resulting game has outcome $\L$, so modulo impartial games, $I$ is greater than or equal to any game.
Hence we have $I \geqslant^-_\II 0$.

To see $b)$, one needs only see that $o^-(I+I^o) = \P$ while $o^-(0+I^o) = \L$.
Hence $I \ngeqslant^-_\DD 0$.
\end{proof}

The second potential extension we considered is: Do we have $G \equiv^-_\II H \Rightarrow G \equiv^-_\DD H$ whenever $G$ and $H$ are dicot?
Unfortunately, we again found counterexamples.

\begin{prop}
\begin{enumerate}[a)]
\item $I \equiv^-_\II \{I|I\}$,
\item $I \not\equiv^-_\DD \{I|I\}$.
\end{enumerate}
\end{prop}

\begin{proof}
By Theorem~\ref{th:IL}, we know that for any impartial game $X$, we have $o^-(I+X) = \L$.
Hence, to prove $a)$, we only need to prove the same for $\{I|I\}$, which we do by induction.
Let $X$ be an impartial game.
From $\{I|I\} + X$, Left can move to $I+X$, which is a mis\`ere $\L$-position.
From $\{I|I\} + X$, Right can either move to $I+X$, a mis\`ere $\L$-position, or to some $\{I|I\} + X'$, which is also a mis\`ere $\L$-position by induction.
Hence $\{I|I\} + X$ is a mis\`ere $\L$-position.

To see $b)$, one need only see that $o^-(I+I^o) = \P$, by Proposition~\ref{prop:adjointsum}, while $o^-(\{I|I\}+I^o) = \N$, both player having a move to $I+I^o$.
Hence $I \not\equiv^-_\DD \{I|I\}$.
\end{proof}

Another potential extension, for which we have no answer yet, would be the following.

\begin{question}
Do we have $G \equiv^-_\II H \Rightarrow G \equiv^-_\DD H$ whenever $G$ is dicot and $H$ is impartial?
\end{question}


The last potential extension we considered was to find a bigger set of games $\UU$ such that $G \equiv^-_\II H \Rightarrow G \equiv^-_\UU H$ whenever $G$ and $H$ are both impartial.
Unfortunately, as Allen pointed out~\cite{mpg}, any universe $\UU$ containing $1 = \{0|\cdot\}$ or $\overline{1} = \{\cdot|0\}$ verifies $*+* \not\equiv^-_\UU 0$.
As $1$ and $\overline{1}$ are the simplest non-dicot position, this could make one think that a set having the required property and strictly containing all dicot positions would not be closed under addition and followers.
However, we here give an example of a universe satisfying these conditions.

First, we prove the following property.

\begin{lem}
\label{lem:IL}
Let $X$ be an impartial game and $n$ a positive integer.
We have $o^-(X+n\{\cdot|I\}) = \L$.
\end{lem}

\begin{proof}
We prove the result by induction on $n$ and $X$.
Assume first Left starts playing in $X + n\{\cdot|I\}$.
If $X$ is not $0$, then Left can play in the $X$ component and leave a mis\`ere $\L$-position by induction.
Otherwise, she cannot play at all and wins immediately.

Assume now Right starts playing in $X + n\{\cdot|I\}$.
If he plays in the $X$ component, he leaves a mis\`ere $\L$-position by induction.
Otherwise, he moves to $X + (n-1)\{\cdot|I\} + I$.
If $n=1$, this is a mis\`ere $\L$-position by Theorem~\ref{th:IL}.
Otherwise, Left can answer to $X + (n-1)\{\cdot|I\} + *$, and as $X+*$ is an impartial game, leave a mis\`ere $\L$-position.

Hence $X + n\{\cdot|I\}$ is a mis\`ere $\L$-position.
\end{proof}

The universe we consider is $\DDI = c\ell(\DD \cup \big\{\{\cdot|I\},\{\overline{I}|\cdot\}\big\})$.
It is closed under addition and followers as it is the closure of a set, and it is closed by conjugates as $\DD$ is closed under conjugates and $\{\cdot|I\}$ and $\{\overline{I}|\cdot\}$ are each other's conjugates.
As $I$ and $\overline{I}$ are dicot games, the only non-dicot games in $\DDI$ are sums of games including $\{\cdot|I\}$ or $\{\overline{I}|\cdot\}$.

We now extend Theorem~\ref{th:I=>D} to comparison modulo $\DDI$.

\begin{theo}
Let $G$ and $H$ be two impartial games such that $G \equiv^-_\II H$.
Then $G\equiv^-_\DDI H$.
\end{theo}

\begin{proof}
For the same reason as in the proof of Theorem~\ref{th:I=>D}, we can consider every option of $G$ is an option of $H$, and every option of $H$ that is not an option of $G$ has $G$ as one of its options.

Now consider a game $X \in \DDI$ such that Left wins $G+X$ playing first (respectively second).
In $H+X$, she can follow the same strategy until either Right plays in $H$, or her strategy recommends a move in $G$, or she has no more moves available in the $X$ component.
In the first case, either he moved the $H$ component to a position $H'$ that is an option of $G$, and she can assume he played that move and resume her strategy, or $G$ is an option of $H'$, so Left can just move the $H$ component to $G$ and resume her strategy.
In the second case, her strategy recommends her to move the $G$ component to some $G'$ that is also an option of $H$, so she can move the $H$ component to $G'$ and resume her strategy.
In the third case, as she has no move in $X$, we have $X = n\{\cdot|I\}$ for some natural integer $n$.
If $n$ is positive, then the position is a mis\`ere $\L$-position by Lemma~\ref{lem:IL}.
If $n=0$, as $G$ and $H$ are $\II$-equivalent, they have the same outcome, hence as Left was winning $G$, she wins $H$ a priori.
Therefore, we have $H \geqslant^-_\DDI G$, and so $H \equiv^-_\DDI G$ by Observation~\ref{obs:imp}.
\end{proof}

Though this universe is somewhat artificial, it is interesting to see that there is still some hope in finding universes bigger than the universe of dicot games, perhaps some not so artificial, sharing this property.


\delete{
This raises the following question.

\begin{question}
Let $\UU$ and $\UU'$ be two sets of games such that $\UU \subset \UU'$ and all ends of $\UU'$ are in $\UU$.
Do we have $G \equiv^-_\UU H \Rightarrow G \equiv^-_\UU H$ whenever $G$ and $H$ are in $\UU$?
\end{question}
}

\end{sloppypar}

\end{document}